\documentclass{amsart}
\usepackage{amsfonts}
\usepackage{amsmath}
\usepackage{amssymb}
\usepackage{amsthm}
\usepackage{graphicx}

\newtheorem{theorem}{Theorem}
\newtheorem{prop}{Proposition}
\newtheorem{lemma}{Lemma}
\newtheorem{rem}{Remark}
\newtheorem{exmp}{Example}
\newtheorem{cor}{Corollary}
\newtheorem{fact}{Fact}

\begin{document}

\title{Zigzag structure of thin chamber complexes}
\author{Michel Deza, Mark Pankov}
\subjclass[2000]{}
\address{Michel Deza: \'Ecole Normale Sup\'erieure, Paris, France}
\email{Michel.Deza@ens.fr}
\noindent
\address{Mark Pankov: Department of Mathematics and Computer Science, 
University of Warmia and Mazury, Olsztyn, Poland}
\email{pankov@matman.uwm.edu.pl}

\begin{abstract}
Zigzags in thin chamber complexes are investigated,
in particular, all zigzags in the Coxeter complexes are described.
Using this description, we show that the lengths of all zigzags in 
the simplex $\alpha_{n}$, the cross-polytope  $\beta_{n}$, the $24$-cell, the icosahedron and the $600$-cell 
are equal to the Coxeter numbers of
$\textsf{A}_{n}$, $\textsf{B}_{n}=\textsf{C}_{n}$, $\textsf{F}_{4}$ and $\textsf{H}_{i}$, $i=3,4$, respectively.
We also discuss in which cases two faces in a thin chamber complex can be connected by a zigzag.
\end{abstract}
\maketitle

\section{Introduction}

In the present paper, we investigate zigzags in {\it thin chamber complexes} which are interesting in the context of  Tits buildings \cite{Tits}.
In fact, these are abstract polytopes, where all facets are simplices.
The well-known Coxeter complexes form an important subclass of thin chamber complexes.
The zigzags in a complex are  the orbits of the action of a special operator $T$ on the set of flags of this complex.
The operator $T$ transfers every flag $F$ to a flag whose faces are adjacent to the faces of $F$. 
The main question under discussion is the following: in which cases can two faces of a complex be connected by a zigzag?

The notion of {\it Petrie polygon} for polytopes is one of the central concepts of famous Coxeter's book \cite{Cox73}.
For embedded graph the same objects are appeared as 
{\it zigzags} in \cite{zig1, zig2, D-DS},  {\em geodesics} in  \cite{GrunMo63} and  {\it left-right paths} in \cite{Sh75}.
Their high-dimensional analogues are considered in \cite{DDS-paper} and \cite{Williams}, see also \cite[Chapter 8]{D-DS}.
Following \cite{zig1,zig2,DDS-paper,D-DS} we call such objects {\it zigzags}.

Consider an abstract polytope ${\mathcal P}$ of rank $n$ and one of its flags $F$.
Let $X_{i}$ be the $i$-face from this flag. There is the unique $i$-face $X'_{i}$ adjacent to $X_{i}$ and incident to all other faces from ${\mathcal P}$.
We define $\sigma_{i}(F)$ as the flag obtained from $F$ by replacing $X_{i}$ on $X'_{i}$ and introduce the operator 
$$T=\sigma_{n-1}\dots\sigma_{0}$$ which acts on the set of all flags.  
In \cite{DDS-paper,D-DS} zigzags are defined as the orbits of this action.
Similarly, for every permutation $\delta$ on the set $\{0,1,\dots, n-1\}$ we consider the operator 
$$T_{\delta}=\sigma_{\delta(n-1)}\dots\sigma_{\delta(0)}$$ and come to {\it generalized zigzags}.
Note that such objects were first considered in \cite{Williams} and named {\it Petrie schemes}.

Now, let us consider the associated flag complex ${\mathfrak F}({\mathcal P})$ which is a thin chamber complex of the same rank $n$.
In Subsection 3.4, we show that there is a natural one-to-one correspondence between  
generalized zigzags in ${\mathcal P}$ and  zigzags in ${\mathfrak F}({\mathcal P})$.
For this reason, it is natural to restrict the general zigzag theory on the case of zigzags in thin chamber complexes
and associated with the operator $T$ only.

We describe all zigzags of Coxeter complexes in terms of Coxeter elements
and show that the length of zigzags depends on the associated Coxeter number (Subsection 3.3).
For example, the Coxeter complexes for $\textsf{A}_{n}$, $\textsf{B}_{n}=\textsf{C}_{n}$, $\textsf{F}_{4}$ and $\textsf{H}_{i}$, $i=3,4$ are 
the flag complexes of the simplex $\alpha_{n}$, the cross-polytope  $\beta_{n}$, the $24$-cell, the icosahedron and the $600$-cell (respectively)
and all zigzags in these complexes are induced by generalized zigzags in the above mentioned polytopes.
In particular, this implies that the lengths of zigzags in the polytopes are equal to the corresponding Coxeter numbers.
On the other hand, the Coxeter systems $\textsf{D}_{n}$ and $\textsf{E}_{i}$, $i=6,7,8$
are related to the half-cube $\frac{1}{2}\gamma_{n}$ and the $\textsf{E}$-polytopes  $2_{21}$, $3_{21}$, $4_{21}$ (respectively).
However, the associated Coxeter numbers are different from the lengths of zigzags in the polytopes. We explain why this occurs.

We say that two faces are $z$-{\it connected} if there is a zigzag joined them. 
If a thin chamber complex is $z$-{\it simple}, 
i.e. every zigzag is not self-intersecting, then the $z$-connectedness of 
any pair of faces implies that the complex is a simplex.
The main result of this part concerns the $z$-connectedness of facets (Theorem \ref{theorem4-1}).
We consider the graph consisting of all facets,  whose edges are pairs intersecting in a ridge.
We determine a class of path geodesics in this graph which can be extended to zigzags.
For $z$-simple complexes, this gives  the full description of path geodesics extendible to zigzags.

\section{Thin simplicial complexes}

\subsection{Definitions and examples}
Let $\Delta$ be an {\it abstract simplicial complex} over a finite set $V$,
i.e. $\Delta$ is formed by subsets of $V$
such that every one-element subset belongs to $\Delta$ 
and for every $X\in \Delta$ all  subsets of $X$ belong to $\Delta$.
Elements of $\Delta$ are called {\it faces} and maximal faces are said to be {\it facets}.
We say that $X\in \Delta$ is a $k$-{\it face} if $|X|=k+1$.
Recall that $0$-faces and $1$-faces are known as {\it vertices} and {\it edges},
respectively, and the empty set is the unique $(-1)$-face.
We will always suppose that the simplicial complex $\Delta$ is {\it pure}, 
i.e. all facets are of the same cardinality $n$.
Then the number $n$ is the {\it rank} of the simplicial complex.

Our second assumption is that  $\Delta$ is {\it thin}.
This means that every ridge, i.e. $(n-2)$-face, is contained in precisely two distinct facets (cf. \cite[Section 1.3]{Tits}).
Two facets are said to be {\it adjacent} if their intersection is a ridge.

Let $k$ be a natural number not greater than $n-2$. 
If a $(k-1)$-face $Y$ is contained in a $(k+1)$-face $Z$,
then there are precisely two $k$-faces $X_{1}$ and $X_{2}$ such that 
$$Y\subset X_{i}\subset Z\;\mbox{ for }\;i=1,2$$
(the set $Z\setminus Y$ contains only two vertices and $X_{1},X_{2}$ are the $k$-faces containing $Y$ and one of these vertices).
We say that two distinct $k$-faces $X_{1}$ and $X_{2}$ are {\it adjacent}
if there exist a $(k-1)$-face $Y$ and a $(k+1)$-face $Z$
satisfying the above inclusion.

For every $k\in \{0,1,\dots,n-1\}$ we denote by $\Gamma_{k}(\Delta)$
the graph whose vertex set consists of all $k$-faces and whose edges are pairs of adjacent faces.
This graph is one of so-called {\it Wythoff kaleidoscopes} \cite{Cox73}.
Following \cite[Section 1.3]{Tits} we say that $\Delta$ is a {\it chamber} complex if the graph $\Gamma_{n-1}(\Delta)$ is connected.
This condition guarantees that $\Gamma_{k}(\Delta)$ is connected for every $k$.
Indeed, for any two $k$-faces $X,Y$ with $k<n-1$ we take facets
$X',Y'$ containing $X,Y$ (respectively);
using a path in $\Gamma_{n-1}(\Delta)$ connecting $X'$ with $Y'$
we construct a path of $\Gamma_{k}(\Delta)$ connecting $X$ with $Y$.

Let us consider a connected simple graph. 
The {\it path distance} $d(v,w)$ between vertices $v$ and $w$ in this graph is the smallest number $d$
such that there is a path of length $d$ connecting these vertices (see, for example, \cite[Section 15.1]{DD}).
Every path of length $d(v,w)$ connecting $v$ and $w$ is called a {\it geodesic}.
The path distance on $\Gamma_{k}(\Delta)$ will be considered in Section 4.

\begin{exmp}\label{exmp2-1}{\rm
The {\it $n$-simplex} $\alpha_{n}$ is the simplicial complex
whose vertex set is the $(n+1)$-element set 
$$[n+1]=\{1,\dots,n+1\}$$
and whose non-empty faces are proper subsets of $[n+1]$.
The {\it cross-polytope} $\beta_{n}$ is the simplicial complex whose vertex set is the set
$$[n]_{\pm}=\{1,\dots,n,-1,\dots,-n\}$$
and whose faces are all subsets $X\subset [n]_{\pm}$ such that for every $i\in X$ we have $-i\not\in X$.
It is clear that $\alpha_{n}$ and $\beta_{n}$ are thin chamber complexes of rank $n$.
Every $\Gamma_{k}(\alpha_{n})$ is the Johnson graph $J(n+1,k+1)$ and 
$\Gamma_{n-1}(\beta_{n})$ is the $n$-dimensional cube graph.
}\end{exmp}

\begin{exmp}\label{exmp2-2}{\rm
Let $\Delta_{1}$ and $\Delta_{2}$ be thin chamber complexes over sets $V_{1}$ and $V_{2}$, respectively. 
The {\it join} $\Delta_{1}*\Delta_{2}$ is the simplicial complex whose vertex set is 
the disjoint union $V_{1}\sqcup V_{2}$ and whose faces are all subsets of type $X_{1}\sqcup X_{2}$, where $X_{i}\in \Delta_{i}$.
This is a thin chamber complex of rank $n_{1}+n_{2}$, where $n_{i}$ is the rank of $\Delta_{i}$.
}\end{exmp}

\begin{exmp}\label{exmp2-3}{\rm
Let ${\mathcal P}$ be a partially ordered set presented as the disjoint union of subsets 
${\mathcal P}_{-1},{\mathcal P}_{0},\dots,{\mathcal P}_{n-1},{\mathcal P}_{n}$ 
such that for any $X\in {\mathcal P}_i$ and $Y\in {\mathcal P}_{j}$ satisfying $X<Y$ we have $i<j$.
The elements of ${\mathcal P}_{k}$ are called $k$-{\it faces}.
There is the unique $(-1)$-face and the unique $n$-face which are the minimal and maximal elements, respectively.
Also, we suppose that every {\it flag}, i.e. a maximal linearly ordered subset,
contains precisely $n$ elements distinct from the $(-1)$-face and $n$-face.
Then ${\mathcal P}$ is an {\it abstract polytope} of rank $n$ if the following conditions hold:
\begin{enumerate}
\item[(P1)] If $k\in \{0,1,\dots,n-1\}$, then 
for any $(k-1)$-face $Y$ and $(k+1)$-face $Z$ satisfying $Y<Z$
there are precisely two $k$-faces $X_{i}$, $i=1,2$ such that $Y<X_{i}<Z$.
\item[(P2)] ${\mathcal P}$ is strongly connected (see, for example, \cite{McM-S}).
\end{enumerate}
Every thin chamber complex of rank $n$ can be considered as an abstract $n$-polytope whose $(n-1)$-faces are $(n-1)$-simplices.
If ${\mathcal P}$ is an abstract $n$-polytope, then the associated {\it flag complex} ${\mathfrak F}({\mathcal P})$
is the simplicial complex whose vertices are the faces of ${\mathcal P}$ and whose facets are the flags.
This is a thin simplicial complex of rank $n$ and (P2) guarantees that ${\mathfrak F}({\mathcal P})$ is a chamber complex.
}\end{exmp}

A simplicial complex is $k$-{\it neighborly} if any $k$ distinct vertices form a face.
Then $\alpha_n$ can be characterized as the unique $n$-neighborly 
thin chamber complex of rank $n$. 
We will use the following fact which follows from a more general result
\cite[p.123]{ConvexPol}.

\begin{fact}\label{fact-neighborly}
If a thin chamber complex of rank $n$ is $k$-neighborly and $k>\lfloor n/2\rfloor$, then it is the $n$-simplex $\alpha_n$.
\end{fact}

\begin{exmp}\label{exmp2-4}{\rm
The join $\alpha_{n}*\alpha_{m}$ with $n\le m$ is a thin chamber complex of rank $n+m$
which is $n$-neighborly and not $(n+1)$-neighborly.
}\end{exmp}

\subsection{Coxeter complexes}
Let $W$ be a finite group generated by a set $S$
whose elements  are involutions and denoted by $s_{1},\dots,s_{n}$.
For any distinct $i,j\in [n]$ we write $m_{ij}$ for the order of the element $s_{i}s_{j}$.
Then $m_{ij}=m_{ji}\ge 2$ and
the condition $m_{ij}=2$ is equivalent to the fact that $s_{i}$ and $s_{j}$ commute.
We suppose that $(W,S)$ is a {\it Coxeter system}, 
i.e $W$ is the quotient  of the free group over $S$
by the normal  subgroup generated by all elements of type $(s_{i}s_{j})^{m_{ij}}$.
The associated {\it diagram}  ${\rm D}(W,S)$ is the graph whose vertex set is $S$
and $s_{i}$ is connected with $s_{j}$ by an edge of order $m_{ij}-2$
(the vertices are disjoint if $m_{ij}=2$).
All finite Coxeter systems are known \cite{Hum}, in particular, every finite irreducible 
Coxeter system is one of the following:
$$\textsf{A}_{n},\;\textsf{B}_{n}=\textsf{C}_{n},\;\textsf{D}_{n},\;
\textsf{F}_{4},\; \textsf{E}_{i},i=6,7,8,\;\textsf{H}_{i},i=3,4,\;
\textsf{I}_{2}(m).$$

For every subset $I=\{i_{1},\dots,i_{k}\}\subset [n]$ we denote by
$W^{I}$ the subgroup generated by the set 
$$S\setminus\{s_{i_{1}},\dots,s_{i_{k}}\}.$$
In particular, for every $i\in [n]$ the subgroup $W^{i}$ is generated
by $S\setminus\{s_{i}\}$.
The following properties are well-known (see, for example, \cite[Section 2.4]{BB}):
\begin{enumerate}
\item[(C1)] $W^{I}\cap W^{J}=W^{I\cup J}$ for any subsets $I,J\subset [n]$,
\item[(C2)] if $v,w\in W$ and $I,J\subset [n]$ then
we have $wW^{I}=vW^{J}$ only in the case when $I=J$ and $w^{-1}v\in W^{I}$.
\end{enumerate}
Also, we will use the following obvious equality 
\begin{equation}\label{eq2-1}
s_{i}W^{j}=W^{j}\;\mbox{ if }\;i\ne j.
\end{equation}

The {\it Coxeter complex} $\Sigma(W,S)$ is the simplicial complex
whose vertices are subsets of type $wW^{i}$ with $w\in W$ and $i\in [n]$.
The vertices $X_{1},\dots,X_{k}$ form a face if 
there exists $w\in W$ such that
$$X_{1}=wW^{i_{1}},\dots,X_{k}=wW^{i_{k}}.$$
This face can be identified with the set 
$$X_{1}\cap\dots\cap X_{k}=wW^{I},\;\mbox{ where }\; 
I=\{i_{1},\dots,i_{k}\}.$$
Every facet is of type 
$\{wW^{1},\dots,wW^{n}\}$
and identified with the element $w$.
So, there is a natural one-to-one correspondence between
facets of $\Sigma(W,S)$ and elements of the group $W$.

\begin{exmp}\label{exmp2-5}{\rm
The Coxeter complexes for
$\textsf{A}_{n}$, $\textsf{B}_{n}=\textsf{C}_{n}$, $\textsf{F}_{4}$ and $\textsf{H}_{i}$, $i=3,4$
are the flag complexes of $\alpha_{n}$, $\beta_{n}$, the $24$-cell,
the icosahedron and the $600$-cell, respectively. 
Note that the $24$-cell is not a simplicial complex.
}\end{exmp}

For every $w\in W$ the left multiplication $L_{w}$ sending $vW^{i}$ to $wvW^{i}$
is an automorphism of the complex $\Sigma(W,S)$.
Also, automorphisms of the diagram ${\rm D}(W,S)$ (if they exist) 
induce automorphisms of $\Sigma(W,S)$.
The automorphism group of $\Sigma(W,S)$ is generated by 
the left multiplications and the automorphisms induced by automorphisms of the diagram
(this statement easily follows from \cite[Corollary 3.2.6]{BB}).

The graph $\Gamma_{n-1}(\Sigma(W,S))$ coincides with the {\it Cayley graph} ${\rm C}(W,S)$ 
whose vertex set is $W$ and elements $w,v\in W$ are adjacent vertices if $v=ws_{i}$ for  a certain $s_{i}\in S$.
Indeed, maximal faces
$\{wW^{1},\dots,wW^{n}\}$ and $\{vW^{1},\dots,vW^{n}\}$
are different precisely in one vertex  if and only if there is a unique 
$i\in [n]$ such that $wW^{i}\ne vW^{i}$, in other words,
$$w^{-1}v\in \bigcap_{j\ne i} W^{j}=\langle s_{i}\rangle$$
and we get the required equality.

\begin{exmp}\label{exmp2-6}{\rm
For the dihedral Coxeter system $\textsf{I}_{2}(m)$ the Cayley graph is the $(2m)$-cycle.
The Cayley graph of $\textsf{A}_{n}$ is the $1$-skeleton of the {\it permutohedron} \cite{Z};
in the general case, the corresponding polytope is called a $W$-{\it permutohedron} \cite{Hohlweg}.
See \cite[Figures 3.3]{BB} for the Cayley graph of $\textsf{H}_{3}$.
}\end{exmp}
 
The {\it length} $l(w)$ of an element $w\in W$
is the smallest length of an expression for $w$ consisting of elements from $S$.
Such an expression is called {\it reduced} if its length is equal to $l(w)$.
Note that elements in a reduced expression are not necessarily mutually distinct.
The path distance between $v,w\in W$ in the Cayley graph ${\rm C}(W,S)$
is equal to $l(v^{-1}w)=l(w^{-1}v)$.

\section{Zigzags in thin simplicial complexes}
In this section, we will always suppose that $\Delta$ is a thin simplicial complex of rank $n$ over a finite set $V$.
Sometimes, $\Delta$ is assumed to be a thin chamber complex.

\subsection{Flags}
Every flag of $\Delta$ can be obtained from a certain sequence of $n$ vertices which form a facet.
Indeed, if  $x_{0},x_{1},\dots,x_{n-1}$ is such a sequence, then the corresponding flag is 
\begin{equation}\label{eq3-1}
\{x_{0}\}\subset\{x_{0},x_{1}\}\subset\dots\subset\{x_{0},x_{1},\dots,x_{n-1}\}.
\end{equation}
Any reenumeration of these vertices gives another flag which contains 
the facet consisting of the vertices.

If $F$ is the flag corresponding to a vertex sequence $x_{0},x_{1},\dots,x_{n-1}$,
then the flag $R(F)$ obtained from the reverse sequence $x_{n-1},\dots, x_{1},x_{0}$
is called the {\it reverse} of $F$, in other words,
the reverse of \eqref{eq3-1}
is the flag
$$\{x_{n-1}\}\subset\{x_{n-1},x_{n-2}\}\subset\dots\subset
\{x_{n-1},\dots,x_{1}\}\subset\{x_{n-1},\dots,x_{1},x_{0}\}.$$
This definition is equivalent to \cite[Definition 8.2]{D-DS}.

Now, let $F$ be the flag formed by $X_{0}, \dots, X_{n-1}$,
where every $X_{i}$ is an $i$-face.
Then for every $i\in \{0,\dots,n-1\}$ there is the unique $i$-face $X'_{i}$
adjacent to $X_{i}$ and incident to all other $X_{j}$.
We denote by $\sigma_{i}(F)$ the flag obtained from $F$
by replacing  $X_{i}$ on $X'_{i}$.
For every flag $F$ we define 
$$T(F):=\sigma_{n-1}\dots\sigma_{1}\sigma_{0}(F).$$
Suppose that $F$ is the flag obtained from a vertex sequence $x_{0},x_{1},\dots,x_{n-1}$.
There is the unique vertex $x_{n}\ne x_{0}$
such that $x_{1},\dots,x_{n}$ form a facet.
An easy verification shows that $T(F)$ is the flag
corresponding to the sequence $x_{1},\dots,x_{n}$.

\begin{lemma}\label{lemma3-1}
For every flag $F$ we have 
$TRT(F)=R(F)$.
\end{lemma}

\begin{proof}
 As above, we suppose that 
$x_{0},x_{1},\dots,x_{n-1}$ and $x_{1},\dots,x_{n}$ are the sequences 
corresponding to the flags $F$ and $T(F)$, respectively.
Then the sequence $x_{n},\dots,x_{1}$ corresponds to the flag $RT(F)$.
This implies that $TRT(F)$ is defined by the sequence $x_{n-1},\dots,x_{1},x_{0}$
and we get the claim.
\end{proof}

Using similar arguments we can prove the following.

\begin{lemma}\label{lemma3-2}
If $A$ is an automorphism of $\Delta$,
then for every flag $F$
we have $AT(F)=TA(F)$.
\end{lemma}

\subsection{Zigzags and their shadows}
For every flag $F$ the sequence 
$$Z=\{T^{i}(F)\}_{i\in{\mathbb N}}$$
(we assume that $0$ belongs to ${\mathbb N}$)
is called a {\it zigzag}.
Since our simplicial complex is finite,
we have $T^{l}(F)=F$ for some $l>0$.
The smallest number $l>0$ satisfying this condition
is said to be the {\it length} of the zigzag.
For any number $i$ the sequence 
$$T^{i}(F),T^{i+1}(F),\dots$$
is also a zigzag; it is obtained from $Z$ by a cyclic permutation of the flags. 
All such zigzags will be identified with $Z$.

The $k$-{\it shadow} of a zigzag $\{F_{i}\}_{i\in {\mathbb N}}$ 
is the sequence $\{X_{i}\}_{i\in {\mathbb N}}$, 
where every $X_{i}$ is the $k$-face from the flag $F_{i}$. 

\begin{prop}\label{prop3-1}
Every zigzag can be uniquely reconstructed from any of the shadows.
\end{prop}

\begin{proof}
Let  $Z=\{F_{i}\}_{i\in {\mathbb N}}$  be a zigzag of length $l$ and 
let $\{X_{i}\}_{i\in {\mathbb N}}$ be the $k$-shadow of $Z$.
The required statement is a consequence of the following two observations.
If $k>0$ then the $(k-1)$-shadow is the sequence
$$X_{l-1}\cap X_{0},\,X_{0}\cap X_{1},\,X_{1}\cap X_{2},\dots.$$
Similarly, if $k< n-1$ then the $(k+1)$-shadow consists of all $X_{i}\cup X_{i+1}$.
\end{proof}

\begin{prop}\label{prop3-2}
If a sequence $Z=\{F_{1},\dots, F_{l}\}$ is a zigzag of length $l$,
then the same holds for the sequence 
$$R(Z)=\{R(F_{l}),R(F_{l-1}),\dots,R(F_{1})\}.$$
\end{prop}

\begin{proof}
If $Z=\{F_{1},\dots, F_{l}\}$  is a zigzag, then 
for every $i\in [l]$ we have $F_{i}=T^{i-1}(F)$, where $F=F_{1}$.
By Lemma \ref{lemma3-1}, 
$$TR(F_{i})=TRT^{i-1}(F)=RT^{i-2}(F)=R(F_{i-1})$$
if $i\ge 2$ and 
$$TR(F_{1})=TRT^{l}(F)=RT^{l-1}(F)=R(F_{l}).$$
This means that $R(Z)$ is a zigzag of length $l$.
\end{proof}

Following \cite[Definition 8.2]{D-DS}, we say 
that the zigzag $R(Z)$ is the {\it reverse} of $Z$.

\begin{rem}\label{rem-shadow}{\rm
Let $Z=\{F_{1},\dots,F_{l}\}$ be a zigzag.
Denote by $F^{k}_{i}$ the $k$-face belonging to the flag $F_{i}$. 
For every $k$ the sequence $F^{k}_{1},\dots,F^{k}_{l}$ 
is the $k$-shadow of $Z$.
The $(n-1)$-shadow of the reverse zigzag is 
$$F^{n-1}_{l},F^{n-1}_{l-1},\dots,F^{n-1}_{1}.$$
Then, by Proposition \ref{prop3-1},
the $(n-2)$-shadow of the reverse zigzag is
$$F^{n-2}_{1},F^{n-2}_{l},F^{n-2}_{l-1},\dots,F^{n-2}_{2}.$$
Step by step, we establish that 
$$F^{n-i-1}_{i},F^{n-i-1}_{i-1},\dots,F^{n-i-1}_{1},F^{n-i-1}_{l},\dots,F^{n-i-1}_{i+1}$$
is the $(n-i-1)$-shadow of the reverse zigzag for every $i$ satisfying $1\le i\le n-1$.
In particular, if $x_{1},\dots,x_{l}$ is the $0$-shadow of $Z$,
then 
$$x_{n-1},x_{n-2},\dots,x_{1},x_{l},x_{l-1},\dots,x_{n}$$
is the $0$-shadow of $R(Z)$.
}\end{rem}

Consider the flag $F$ defined by a vertex sequence $x_{0},x_{1},\dots,x_{n-1}$, i.e. 
$$\{x_{0}\}\subset \{x_{0},x_{1}\}\subset\dots\subset \{x_{0},x_{1},\dots, x_{n-1}\}.$$
There is the unique facet containing $x_{1},\dots,x_{n-1}$
and distinct from the facet of $F$.
In this facet, we take the unique vertex $x_{n}$ 
distinct from $x_{1},\dots,x_{n-1}$. 
It was noted above that $T(F)$ is the flag
$$\{x_{1}\}\subset \{x_{1},x_{2}\}\subset\dots\subset \{x_{1},\dots, x_{n}\}.$$
We apply the same arguments to the latter flag and get a certain vertex $x_{n+1}$.
Recurrently, we construct a sequence of vertices $\{x_{i}\}_{i\in {\mathbb N}}$
such that  $T^{i}(F)$ is the flag
\begin{equation}\label{eq3-2}
 \{x_{i}\}\subset \{x_{i},x_{i+1}\}\subset\dots\subset \{x_{i},\dots,x_{i+n-1}\}.
\end{equation}
The sequence $\{x_{i}\}_{i\in {\mathbb N}}$ 
is the $0$-shadow of the zigzag $\{T^{i}(F)\}_{i\in {\mathbb N}}$.
For every $i\in {\mathbb N}$ the following assertions are fulfilled:
\begin{enumerate}
\item[(Z1)]  $x_{i}, x_{i+1},\dots, x_{i+n-1}$
form a facet,
\item[(Z2)] $x_{i}\ne x_{n+i}$.
\end{enumerate}
If $l$ is the length of the zigzag,
then $l>n$ and $x_{i+l}=x_{i}$ for all $i\in {\mathbb N}$.
The equality $x_{i}=x_{j}$ is possible for some distinct $i,j\in \{0,1,\dots,l-1\}$
only in the case when $|i-j|>n$.
We say that the zigzag $\{T^{i}(F)\}_{i\in {\mathbb N}}$ is {\it simple} if  
$x_{0},x_{1}.\dots,x_{l-1}$ are mutually distinct.
In this case, the $k$-shadow is formed by $l$ mutually distinct $k$-faces for every $k$.
The complex $\Delta$ is said to be $z$-{\it simple}
if every zigzag is simple.

\begin{prop}\label{prop3-3}
If $\{x_{i}\}_{i\in {\mathbb N}}$ is a sequence of vertices
satisfying {\rm (Z1)} and {\rm (Z2)} for every $i$, then all flags of type 
{\rm \eqref{eq3-2}}
form a zigzag and $\{x_{i}\}_{i\in {\mathbb N}}$ is the $0$-shadow of this zigzag.
\end{prop}

\begin{proof}
Easy verification.
\end{proof}

Suppose that a zigzag $Z$ is the reverse of itself, i.e. $R(Z)$
can be obtained from $Z$ by a cyclic permutation of the flags.
Then the $0$-shadow of $Z$  is a sequence of type 
$$x_{0},x_{1},\dots, x_{n-1}, x_{n},\dots,x_{m},\dots,x_{m},\dots,x_{n},x_{n-1},\dots,x_{1},x_{0},\dots$$
and for a sufficiently large $m$ the distance between two exemplars of $x_{m}$ is not greater than $n$. 
This contradicts (Z1).
So, every zigzag is not the reverse of itself.
In what follows, {\it every zigzag will be identified with its reverse}.

We say that $\Delta$ is $z$-{\it uniform} if all zigzags are of the same length. 

\begin{lemma}\label{lemma-numberzigzags}
If $\Delta$ is $z$-uniform and the length of zigzags is equal to $l$, then 
there are precisely $n!N/2l$ zigzags, where $N$ is the number of facets in $\Delta$.
\end{lemma}

\begin{proof}
By the definition, zigzags are the orbits of the action of the operator $T$ on the set of flags.
Thus, the sum of the lengths of all zigzags is equal to the number of flags.
There are precisely $n!N$ distinct flags. Since every zigzag is identified with its reverse,
we get precisely $n!N/2l$ zigzags.
\end{proof}

\begin{exmp}\label{exmp3-1}{\rm
There is a natural one-to-one correspondence between
zigzags of $\alpha_{n}$ and permutation on the set $[n+1]$.
Every zigzag is of length $n+1$ and Lemma \ref{lemma-numberzigzags} implies that
the number of zigzag is equal to $\frac{n!}{2}$.
}\end{exmp}

There is the following characterizations of the $n$-simplex in terms of the length of zigzags.

\begin{prop}\label{prop3-4}
Suppose that $\Delta$ is a thin chamber complex.
Then it contains a zigzag of length $n+1$ if and only if it is the $n$-simplex $\alpha_{n}$.
\end{prop}

\begin{proof}
By Example \ref{exmp3-1}, every zigzag in $\alpha_{n}$ is of length $n+1$. 
Conversely, suppose that $Z$ is a zigzag of length $n+1$. 
It follows from (Z1) and (Z2) that 
the $0$-shadow of $Z$ consists of $n+1$ mutually distinct vertices.
Denote by $X$ the subset of $V$ formed by these $n+1$ vertices.
By (Z1), every $n$-element subset of $X$
is a facet of $\Delta$.
Since every $(n-1)$-element subset of $X$ is the intersection
of two such facets,
the connectedness of $\Gamma_{n-1}(\Delta)$ guarantees that 
there are no other facets in $\Delta$.
\end{proof}

\begin{exmp}\label{exmp3-2}{\rm
The $0$-shadow of every zigzag in $\beta_{n}$ is a sequence of the following type
$$i_{1},\dots,i_{n},-i_{1},\dots,-i_{n},$$ where $i_{1},\dots,i_{n}$ form a facet.
Thus, all zigzags are of length $2n$. 
Since $\beta_{n}$ has precisely $2^{n}$ facets, 
Lemma \ref{lemma-numberzigzags} shows that there are precisely $2^{n-2}(n-1)!$ zigzags.
}\end{exmp}

The $0$-shadows of  zigzags considered in Examples \ref{exmp3-1} and \ref{exmp3-2} 
consist of all vertices of the complex. For the general case this fails.

By Lemma \ref{lemma3-2}, every automorphism of $\Delta$ sends zigzags to zigzags. 
We say that $\Delta$ is $z$-{\it transitive} if for any two zigzags there is an automorphism of $\Delta$
transferring one of them to the other.
The complexes $\alpha_{n}$, $\beta_{n}$ are $z$-transitive.
Also, they are $z$-simple.

\subsection{Zigzags in Coxeter complexes}
Let $(W,S)$ be, as in Subsection 2.2, a finite Coxeter system and let
$s_{1},\dots,s_{n}$ be the elements of $S$.
For every permutation $\delta$ on the set $[n]$ we consider 
the corresponding {\it Coxeter element}
$$s_{\delta}:=s_{\delta(1)}\dots s_{\delta(n)}.$$
It is well-known that the order of this element does not depend on $\delta$
\cite[Section 3.16]{Hum}.
This order is denoted by $h$ and called the {\it Coxeter number}.
All Coxeter numbers are known \cite[p.80, Table 2]{Hum}.
Also, we denote by $E_{\delta}$ the flag in $\Sigma(W,S)$ 
obtained from the following sequence of vertices
$$W^{\delta(1)},\dots,W^{\delta(n)}.$$
The facet in this flag is identified with the identity element $e$.
The second facet containing $W^{\delta(2)},\dots, W^{\delta(n)}$
is 
$$\{s_{\delta(1)}W^{\delta(2)}=W^{\delta(2)},\dots,s_{\delta(1)}W^{\delta(n)}=W^{\delta(n)},
s_{\delta(1)}W^{\delta(1)}\}.$$
This facet is identified with $s_{\delta(1)}$.
So, the flag $T(E_{\delta})$ is defined by the sequence
$$W^{\delta(2)},\dots, W^{\delta(n)},s_{\delta(1)}W^{\delta(1)}.$$
The facet containing $W^{\delta(3)},\dots,W^{\delta(n)},s_{\delta(1)}W^{\delta(1)}$
and distinct from $s_{\delta(1)}$ is 
$$\{s_{\delta(1)}s_{\delta(2)}W^{\delta(3)}=W^{\delta(3)},\dots,
s_{\delta(1)}s_{\delta(2)}W^{\delta(1)}=s_{\delta(1)}W^{\delta(1)}, 
s_{\delta(1)}s_{\delta(2)}W^{\delta(2)}\}.$$
This facet is identified with $s_{\delta(1)}s_{\delta(2)}$
and the flag $T^{2}(E_{\delta})$ is related to 
the vertex sequence
$$W^{\delta(3)},\dots, W^{\delta(n)},s_{\delta(1)}W^{\delta(1)},
s_{\delta(1)}s_{\delta(2)}W^{\delta(2)}.$$
Similarly, we show that for every $i\in[n-1]$ the flag $T^{i}(E_{\delta})$
is defined by the sequence
$$W^{\delta(i+1)},\dots,W^{\delta(n)},s_{\delta(1)}W^{\delta(1)},
s_{\delta(1)}s_{\delta(2)}W^{\delta(2)},\dots,
s_{\delta(1)}\dots s_{\delta(i)}W^{\delta(i)}$$
and the flag 
$T^{n}(E_{\delta})$ corresponds to the sequence
$$s_{\delta(1)}W^{\delta(1)},s_{\delta(1)}s_{\delta(2)}W^{\delta(2)},\dots,
s_{\delta(1)}\dots s_{\delta(n)}W^{\delta(n)}.$$
Using \eqref{eq2-1} we rewrite the latter sequence as follows
$$s_{\delta}W^{\delta(1)},\dots,s_{\delta}W^{\delta(n)}.$$
Therefore,
$$T^{n}(E_{\delta})=L_{s_{\delta}}(E_{\delta})$$
(recall that $L_{w}$ is the left multiplication sending every $vW^{i}$ to $wvW^{i}$, see Subsection 2.2).
Since $L_{s_{\delta}}$ is an automorphism of $\Sigma(W,S)$,
Lemma \ref{lemma3-2} implies that
$$T^{n+i}(E_{\delta})=L_{s_{\delta}}T^{i}(E_{\delta})$$
for every $i\in {\mathbb N}$, in particular,
$$T^{mn}(E_{\delta})=L_{s^{m}_{\delta}}(E_{\delta}).$$
Thus the length of the zigzag $\{T^{i}(E_{\delta})\}_{i\in {\mathbb N}}$ 
is equal to $nh$ and the $0$-shadow is 
$$W^{\delta(1)},\dots,W^{\delta(n)},s_{\delta}W^{\delta(1)},\dots,s_{\delta}W^{\delta(n)},
\dots,
s^{h-1}_{\delta}W^{\delta(1)},\dots,s^{h-1}_{\delta}W^{\delta(n)}.$$
The $(n-1)$-shadow of this zigzag is the following 
$$e,s_{\delta(1)},s_{\delta(1)}s_{\delta(2)},\dots, s_{\delta},s_{\delta}s_{\delta(1)},\dots,
s^{2}_{\delta}, s^{2}_{\delta}s_{\delta(1)},\dots,s^{h}_{\delta}=e.$$
For every number $m<h$ the element $s^{m}_{\delta}$ does not belong to any $W^{i}$ (see \cite[Theorem 3.1]{Williams}).
This implies that the zigzag $\{T^{i}(E_{\delta})\}_{i\in {\mathbb N}}$ is simple.

For every flag $F$ in $\Sigma(W,S)$ 
there exist $w\in W$ and a permutation $\delta$ on the set $[n]$
such that $F=L_{w}(E_{\delta})$. The automorphism $L_{w}$
sends the zigzag $\{T^{i}(E_{\delta})\}_{i\in {\mathbb N}}$
to the zigzag $\{T^{i}(F)\}_{i\in {\mathbb N}}$. 
The $0$-shadow of the latter zigzag is
\begin{equation}\label{eq3-3}
wW^{\delta(1)},\dots,wW^{\delta(n)},\dots,
ws^{h-1}_{\delta}W^{\delta(1)},\dots,ws^{h-1}_{\delta}W^{\delta(n)}
\end{equation}
and the $(n-1)$-shadow is
\begin{equation}\label{eq3-4}
w,ws_{\delta(1)},ws_{\delta(1)}s_{\delta(2)},\dots, ws_{\delta},ws_{\delta}s_{\delta(1)},\dots,
ws^{2}_{\delta}, ws^{2}_{\delta}s_{\delta(1)},\dots,ws^{h}_{\delta}=w.
\end{equation}
Since $\Sigma(W,S)$ is $z$-uniform and contains precisely $|W|$ facets, 
we can find the number of zigzags in $\Sigma(W,S)$ using Lemma \ref{lemma-numberzigzags}.

So, we get the following. 

\begin{prop}\label{prop3-5}
The following assertions are fulfilled:
\begin{enumerate}
\item[(1)] the Coxeter complex $\Sigma(W,S)$ is $z$-simple;
\item[(2)] there are precisely $|W|(n-1)!/2h$ distinct zigzags in $\Sigma(W,S)$ and the length of every zigzag is equal to $nh$, 
where $h$ is the Coxeter number corresponding to $(W,S)$ and $n=|S|$;
\item[(3)]
the $0$-shadow and the $(n-1)$-shadow of every zigzag in $\Sigma(W,S)$
are described by the formulas \eqref{eq3-3} and \eqref{eq3-4}, respectively.
\end{enumerate}
\end{prop}

The existence of an automorphism of $\Sigma(W,S)$ 
transferring the flag defined by the sequence $W^{1},\dots,W^{n}$
to the flag associated to the sequence $W^{\delta(1)},\dots,W^{\delta(n)}$
is equivalent to the fact that the permutation $\delta$ induces an automorphism of 
the diagram ${\rm D}(W,S)$. So, $\Sigma(W,S)$ is $z$-transitive only in some special cases
when every permutation on the set $[n]$ induces an automorphism of the diagram.

\subsection{Generalized zigzags}
Let $\delta$ be a permutation on the set $\{0,1,\dots,n-1\}$.
Consider the operator 
$$T_{\delta}:=\sigma_{\delta(n-1)}\dots\sigma_{\delta(1)}\sigma_{\delta(0)}$$
on the set of flags in $\Delta$.
For every flag $F$ the sequence
$\{T^{i}_{\delta}(F)\}_{i\in{\mathbb N}}$
will be called a $\delta$-{\it zigzag}.
Also, we say that this is a {\it generalized zigzag}. 
As in Subsection 3.2, we define the {\it length} and {\it shadows}
of generalized zigzags. 

\begin{rem}{\rm
Let $F$ be a flag in $\Delta$ and let 
$Z=\{F,T(F),\dots,T^{l-1}(F)\}$ be the associated zigzag.
The operator 
$\tilde{T}=\sigma_{0}\sigma_{1}\dots\sigma_{n-1}$ 
coincides with $T^{-1}$ and we have
$$\tilde{T}(F)=T^{l-1}(F),$$
$$\tilde{T}^{2}(F)=T^{2l-2}(F)=T^{l-2}(F),$$
$$\vdots$$
$$\tilde{T}^{l-1}(F)=T^{(l-1)^2}(F)=T(F).$$
Therefore, 
$\{F,\tilde{T}(F),\dots,\tilde{T}^{l-1}(F)\}$ is the zigzag reversed to $Z$.
Similarly, we show that the generalized zigzags defined by the operator 
$\sigma_{\delta(0)}\sigma_{\delta(1)}\dots\sigma_{\delta(n-1)}$ are 
reversed to the generalized zigzags obtained from $T_{\delta}$.
}\end{rem}

Let $F$ be a flag of $\Delta$ whose $k$-face is denoted by $X_{k}$
for every $k\in \{0,1,\dots,n-1\}$. 
This is a facet in the flag complex ${\mathfrak F}(\Delta)$
and we consider the zigzag $Z$ in ${\mathfrak F}(\Delta)$
defined by the vertex sequence 
$$X_{\delta(0)},X_{\delta(1)},\dots,X_{\delta(n-1)},$$
where $\delta$ is a certain permutation on the set $\{0,1,\dots,n-1\}$.
Let $\{Y_{i}\}_{i\in {\mathbb N}}$ be the $0$-shadow of this zigzag.
Then
$$Y_{0}=X_{\delta(0)},Y_{1}=X_{\delta(1)},\dots,Y_{n-1}=X_{\delta(n-1)}$$
and $Y_{n}$ is the $\delta(0)$-face in the flag $\sigma_{\delta(0)}(F)$.
Similarly, $Y_{n+1}$ is the $\delta(1)$-face in the flag $\sigma_{\delta(1)}\sigma_{\delta(0)}(F)$.
Step by step, we establish that $Y_{n+i}$ is 
the $\delta(i)$-face in the flag 
$$
\sigma_{\delta(i)}\dots\sigma_{\delta(1)}\sigma_{\delta(0)}(F)
$$
for every $i\in \{0,1,\dots,n-1\}$.
Therefore, $Y_{n},Y_{n+1},\dots,Y_{2n-1}$ belongs to the flag
$T_{\delta}(F)$. 
The same arguments show that 
$Y_{kn+i}$ is the $\delta(i)$-face in the flag $T^{k}_{\delta}(F)$
for every $k\in {\mathbb N}$ and $i\in \{0,1,\dots,n-1\}$.
In other words, the $0$-shadow of $Z$ is formed by 
the faces from the $\delta$-zigzag $\{T^{i}_{\delta}(F)\}_{i\in {\mathbb N}}$, i.e.
$$\underbrace{Y_{0},Y_{1},\dots,Y_{n-1}}_{F},
\underbrace{Y_{n},\dots,Y_{2n-1}}_{T_{\delta}(F)},
\underbrace{Y_{2n},\dots,Y_{3n-1}}_{T^{2}_{\delta}(F)},\dots\,.$$
Thus the length of $Z$ is equal to $nl$, 
where $l$ is the length of $\{T^{i}_{\delta}(F)\}_{i\in {\mathbb N}}$.
It is trivial that $Z$ is simple if and only if $\{T^{i}_{\delta}(F)\}_{i\in {\mathbb N}}$
is simple. In particular, we have proved the following.

\begin{prop}\label{prop3-6}
There is a natural one-to-one correspondence between 
zigzags in ${\mathfrak F}(\Delta)$ and generalized zigzags in $\Delta$.
The length of a zigzag in ${\mathfrak F}(\Delta)$ is equal to $nl$,
where $l$ is the length of the corresponding generalized zigzag in $\Delta$.
\end{prop}

Lemma \ref{lemma-numberzigzags} implies the following.

\begin{cor}
If ${\mathfrak F}(\Delta)$ is $z$-uniform and the length of generalized zigzags in $\Delta$ is equal to $l$, then 
there are precisely $(n-1)!N/2l$ generalized zigzags, where $N$ is the number of flags in $\Delta$.
\end{cor}

Clearly, generalized zigzags can be defined in abstract polytopes and Proposition \ref{prop3-6} holds  for this case.
Recall that the flag complexes of $\alpha_{n}$, $\beta_{n}$,
the $24$-cell, the icosahedron and the $600$-cell are 
the Coxeter complexes of $\textsf{A}_{n}$, $\textsf{B}_{n}=\textsf{C}_{n}$,
$\textsf{F}_{4}$ and $\textsf{H}_{i}$, $i=3,4$ (respectively).
Propositions \ref{prop3-5} and \ref{prop3-6} imply the following.

\begin{cor}
The lengths of generalized zigzags in  $\alpha_{n}$, $\beta_{n}$,
the $24$-cell, the icosahedron and the $600$-cell are  equal to 
the corresponding Coxeter numbers
$$h(\textsf{A}_{n})=n+1,\; h(\textsf{B}_{n})=2n,\; 
h(\textsf{F}_{4})=12,\;h(\textsf{H}_{3})=10,\;h(\textsf{H}_{4})=30,$$
respectively.
\end{cor}

Note that all values from Corollary 1 were given in \cite{DDS-paper,D-DS}, but 
the connection with Coxeter numbers is new.

Now, we consider the half-cube $\frac{1}{2}\gamma_{n}$
and the $\textsf{E}$-polytopes $2_{21}$, $3_{21}$, $4_{21}$
associated to the Coxeter systems $\textsf{D}_{n}$ and 
$\textsf{E}_i$, $i=6,7,8$ (respectively).
The zigzag lengths of these polytopes
are not equal to the corresponding Coxeter numbers.
The Coxeter numbers of $\textsf{E}_i$, $i=6,7,8$, are $12,18, 30$ and, by \cite{DDS-paper,D-DS},
the zigzag lengths of the ${\textsf E}$-polytopes are $18,90,36$.
The Coxeter number of $\textsf{D}_{n}$ is equal to $2(n-1)$
and the zigzag length of $\frac{1}{2}\gamma_{n}$ can be found in \cite{DDS-paper,D-DS} for $n\le 13$.

The following example explains this non-coincidence.

\begin{exmp}{\rm
For every $i\in [n-1]$ we denote by ${\mathcal F}_{i}$ the set of all $i$-faces in $\beta_{n}$. 
Then ${\mathcal F}_{n-1}$ can be presented as the disjoint union of two subsets 
${\mathcal F}_{+}$ and ${\mathcal F}_{-}$ satisfying the following condition:
for any $X,Y\in {\mathcal F}_{s}$, $s\in \{+,-\}$ 
the number  $n-|X\cap Y|$ is even and 
this number is odd if $X\in {\mathcal F}_{+}$ and $Y\in {\mathcal F}_{-}$.
Suppose that $(W,S)$ is the Coxeter system of type $\textsf{D}_{n}$ (the corresponding diagram is on the figure below).
\begin{center}
\vspace{1ex}
\includegraphics[width=0.4\textwidth]{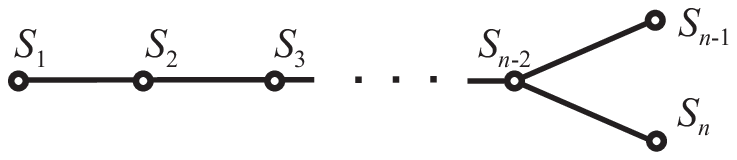}
\vspace{1ex}
\end{center}
Recall that the vertices of the Coxeter complex $\Sigma(W,S)$
are all subsets $wW^{i}$, $w\in W$ and $i\in [n]$.
For every $i\in [n-2]$ the vertices of type $wW^{i}$ correspond to the elements of ${\mathcal F}_{i-1}$
and the vertices of types $wW^{n-1}$ and $wW^{n}$ correspond to 
the elements of ${\mathcal F}_{+}$ and ${\mathcal F}_{-}$ (respectively),
see \cite[Example 2.7]{Pankov-book} for the details.
For every $i\in [n-1]$ the $i$-faces of the hypercube $\gamma_{n}$ 
can be identified with the elements of ${\mathcal F}_{n-1-i}$.
In the case when the vertex set of the half-cube $\frac{1}{2}\gamma_{n}$ is ${\mathcal F}_{+}$,
the $2$-faces of $\frac{1}{2}\gamma_{n}$ are the element of 
${\mathcal F}_{n-3}\cup {\mathcal F}_{-}$.
In other words, the $2$-faces of $\frac{1}{2}\gamma_{n}$
give two different types of vertices in the Coxeter complex. 
This means that zigzags in the Coxeter complex of $\textsf{D}_{n}$ cannot be obtained from
generalized zigzags in $\frac{1}{2}\gamma_{n}$ as it was described in Proposition \ref{prop3-6}.
For the same reason zigzags in the Coxeter complexes of $\textsf{E}_i$, $i=6,7,8$
are not induced by generalized zigzags of the associated $\textsf{E}$-polytopes.
}\end{exmp}

\section{$Z$-connectedness of faces}
In this section, we suppose that $\Delta$ is a thin chamber complex of rank $n$.
Then every $\Gamma_{k}(\Delta)$ is connected.
We say that two faces $X$ and $Y$  in $\Delta$ are $z$-{\it connected} if there is a zigzag 
such that $X$ and $Y$ are contained in some flags of this zigzag.

\begin{exmp}\label{exmp4-1}{\rm
In the $n$-simplex $\alpha_{n}$, any two faces are $z$-connected.
This easily follows from the fact that every zigzag is defined by a certain enumeration of the set $[n+1]$
(Example \ref{exmp3-1}).
}\end{exmp}

\begin{exmp}\label{exmp4-2}{\rm
If $n\ge 3$, then the complex $\beta_{n}$ contains pairs of faces which are not $z$-connected.
Consider, for example, the edges
$$\{i,j\}\;\mbox{ and }\;\{i,-j\}.$$
Up to a cyclic permutation, the $0$-shadow of a zigzag containing $\{i,j\}$ is of type
$$i,j,i_{1},\dots,i_{n-2},-i,-j,-i_{1},\dots,-i_{n-2},$$
where each of $i_{1},\dots,i_{n-2}$ is not equal to $\pm i$ or $\pm j$ (see Example \ref{exmp3-2}).
The corresponding zigzag does not contain $\{i,-j\}$.
Similarly, we can show that the $k$-faces
$$\{i_{1},\dots,i_{k},j\}\;\mbox{ and }\;\{i_{1},\dots,i_{k},-j\}$$
are not $z$-connected if $k<n-1$.
Recall that $\Gamma_{n-1}(\beta_{n})$ is the $n$-dimensional cube graph.
It is easy to see that every geodesic of this graph is contained in the $(n-1)$-shadow of a certain zigzag.
}\end{exmp}

\subsection{$Z$-connectedness of facets and distance normal geodesics}
If $\Delta$ is $\alpha_{n}$ or $\beta_{n}$, then
\begin{equation}\label{eq4-1}
d(X,Y)=n- |X\cap Y|
\end{equation}
for any two facets $X,Y$. Recall that $d(X,Y)$ is the path distance between $X$ and $Y$ in $\Gamma_{n-1}(\Delta)$.
For the general case this distance formula fails.
Consider, for example, an $n$-gonal bipyramid;
if $n\ge 6$, then it contains two faces $X,Y$ with a common vertex and such that $d(X,Y)>2$. 
In the general case, we have 
$$d(X,Y)\ge n-|X\cap Y|$$
for any two facets $X,Y$ in $\Delta$.

We say that facets $X$ and $Y$ satisfying $d(X,Y)\le n$ form a {\it distance normal pair} if the equality \eqref{eq4-1} holds.
Any two adjacent facets form a distance normal pair.
If the path distance between two facets is equal to $2$, then they form a distance normal pair.
It was noted above that any two facets in $\alpha_{n}$ or $\beta_{n}$ form a distance normal pair.
Every geodesic joining distance normal pair $X,Y$ will be called {\it distance normal geodesics} if $d(X,Y)\le n$.

\begin{exmp}\label{exmp4-3}{\rm
Let $(W,S)$ be a Coxeter system and $|S|=n$.
Consider the associated Coxeter complex $\Sigma(W,S)$.
By Subsection 2.2,
the graph $\Gamma_{n-1}(\Sigma(W,S))$ can be naturally identified with the Cayley graph ${\rm C}(W,S)$. 
Let $s_{i}$ and $s_{j}$ be non-com\-muting elements of $S$.
The path distance between the identity element $e$ and $w=s_{i}s_{j}s_{i}$ in ${\rm C}(W,S)$ is equal to $3$.
On the other hand, $e$ and $w$ correspond to the facets
$\{W^{1},\dots,W^{n}\}$ and $\{wW^{1},\dots,wW^{n}\}$,
respectively.
Since we have $wW^{k}=W^{k}$ for every $k\ne i,j$, the intersection of these facets
contains precisely $n-2$ vertices and they do not form a distance normal pair.
}\end{exmp}

The latter example can be generalized as follows.

\begin{exmp}\label{exmp4-4}{\rm
Let $w$ and $v$ be elements of $W$ such that $l(w^{-1}v)\le n$. 
Then $d(w,v)\le n$.
The facets of $\Sigma(W,S)$ corresponding to $w$ and $v$
form a distance normal pair if and only if there is a reduced expression for $w^{-1}v$
whose elements are mutually distinct
(the existence of such an expression implies that all reduced expressions of $w^{-1}v$
satisfy the same condition).
}\end{exmp}

Now, we consider facets $X$ and $Y$ such that $d(X,Y)>n$.
We say that $X$ and $Y$ form a {\it distance normal pair} if there exists a geodesic
$$
 X=X_{0},X_{1},\dots,X_{m}=Y,
$$
where any two $X_{i},X_{j}$ satisfying $d(X_{i},X_{j})\le n$ form a distance normal pair.
Every such a geodesic will be called {\it distance normal}.
The fact that two facets are connected by a distance normal geodesic does not guarantee that 
every geodesic connecting them is distance normal.

\begin{exmp}{\rm
It is not difficult to construct a thin chamber complex of rank $3$ satisfying the following conditions:
\begin{enumerate}
\item[$\bullet$] there is a distance normal pair of faces $X,Y$ such that $d(X,Y)=4$,
\item[$\bullet$] there is a face $X'$ adjacent to $X$ and intersecting $Y$ precisely in a vertex,
\item[$\bullet$] $d(X',Y)=3$. 
\end{enumerate}
So, $X'$ and $Y$ do not form a distance normal pair.
Let $X',X_{1},X_{2},Y$ be a geodesic connecting $X'$ and $Y$.
Then $X,X',X_{1},X_{2},Y$ is a geodesic from $X$ to $Y$ which  is not distance normal. 
}\end{exmp}

The $(n-1)$-shadows of zigzags in $\Delta$ are closed (not necessarily simple) paths in $\Gamma_{n-1}(\Delta)$.
We say that a path $X_{1},\dots,X_{m}$ is contained in a path $Y_{1},\dots,Y_{k}$
(or $Y_{1},\dots,Y_{k}$ contains $X_{1},\dots,X_{m}$) if 
there is a number $j$ such that $Y_{j+i}=X_{i}$ for every $i\in \{0,1,\dots,m\}$. 
It is clear that every path contained in a distance normal geodesic is a distance normal geodesic.

\begin{lemma}\label{lemma4-1}
If $Z$ is a simple zigzag in $\Delta$, then every geodesic of $\Gamma_{n-1}(\Delta)$ contained in the $(n-1)$-shadow of $Z$  is distance normal.
\end{lemma}

\begin{proof}
Easy verification.
\end{proof}

\begin{theorem}\label{theorem4-1}
Every distance normal geodesic of $\Gamma_{n-1}(\Delta)$ is contained in the $(n-1)$-shadow of a certain zigzag of $\Delta$
and the following assertions are fulfilled: 
\begin{enumerate}
\item[(1)] if this geodesic is of length $m\le n$, then  there are at most $(n-m)!$ zigzags whose $(n-1)$-shadows contain the geodesic;
\item[(2)]
if the length of the geodesic is greater than $n$, then it is contained in the $(n-1)$-shadow of the unique zigzag.
\end{enumerate}
\end{theorem}

Theorem \ref{theorem4-1} together with Lemma \ref{lemma4-1} give the following.

\begin{cor}\label{cor4-1}
Suppose that $\Delta$ is $z$-simple. 
A geodesic of $\Gamma_{n-1}(\Delta)$  is contained in the $(n-1)$-shadow of a certain zigzag if and only if it is distance normal.
\end{cor}

\begin{rem}{\rm
Example \ref{exmp4-4} shows that 
for the Coxeter complexes Theorem \ref{theorem4-1} easily follows from
Proposition \ref{prop3-5}.
}\end{rem}

\subsection{Proof of Theorem \ref{theorem4-1}}
Let $X_{0},X_{1},\dots,X_{m}$ be a distance normal geodesic in $\Gamma_{n-1}(\Delta)$.
First, we consider the case when $m\le n$ and prove the statement by induction.

If $m=1$,  then we take vertices $x_{0},x_{1},\dots,x_{n}$ such that 
$$X_{0}=\{x_{0},x_{1},\dots,x_{n-1}\}\;\mbox{ and }\;X_{1}=\{x_{1},\dots,x_{n}\}.$$
For every permutation $\delta$ on the set $[n-1]$ 
we consider the sequence
$$x_{0},x_{\delta(1)},\dots,x_{\delta(n-1)}$$
and denote by $F_{\delta}$ the associated flag
$$\{x_{0}\}\subset \{x_{0},x_{\delta(1)}\}\subset\dots\subset 
\{x_{0},x_{\delta(1)},\dots,x_{\delta(n-1)}\}=X_{0}.$$
Since the vertices $x_{\delta(1)},\dots,x_{\delta(n-1)}$ belong to $X_{1}$,
the facet of the flag $T(F_{\delta})$ is $X_{1}$.
It is easy to see that every zigzag whose $(n-1)$-shadow contains the path $X_{0},X_{1}$ is of 
type $\{T^{i}(F_{\delta})\}_{i\in {\mathbb N}}$.
Since $\Delta$ is not assumed to be $z$-simple,
the zigzags corresponding to the distinct flags $F_{\delta}$ and $F_{\gamma}$ may be coincident
or the zigzag defined by $F_{\delta}$ is the reverse of the zigzag defined by $F_{\gamma}$.
Therefore, there are at most $(n-1)!$ zigzags satisfying the required condition.

Let $m>1$. Then $X_{0}, X_{1},\dots X_{m-1}$ is  a distance normal geodesic and,
by the inductive hypothesis, it is contained in the $(n-1)$-shadow of a certain zigzag $Z$. 
Let $\{x_{i}\}_{i\in {\mathbb N}}$ be the $0$-shadow of $Z$.
We suppose that the vertices $x_{0},x_{1},\dots,x_{n-1}$ belong to $X_{0}$.
Then 
$$X_{i}=\{x_{i},\dots,x_{i+n-1}\}$$
for every $i\in [m-1]$.

If $n=m$, then $X_{0}\cap X_{m}=\emptyset$, in particular, $x_{n-1}\not\in X_{n}$.
Since 
$$X_{n-1}=\{x_{n-1},\dots,x_{2n-2}\}$$ 
and $X_{n}$ are adjacent,
the vertices $x_{n},\dots,x_{2n-2}$ belong to $X_{n}$.
This implies that $X_{n}$ coincides with the facet
$\{x_{n},\dots,x_{2n-1}\}$ and the geodesic $X_{0},X_{1},\dots,X_{n}$ is contained in 
the $(n-1)$-shadow of $Z$.  
This is the unique zigzag whose $(n-1)$-shadow contains this geodesic. 
Indeed, every $x_{i}$, $i\in \{0,1,\dots, n-1\}$ is the unique vertex belonging to $X_{i}\setminus X_{i+1}$
and $Z$ is completely determined by the list of the first $n$ vertices in the $0$-shadow.

Consider the case when $m<n$.
The face 
$$A=\{x_{m-1},\dots,x_{n-1}\}$$
consists of $n-m+1\ge 2$ vertices and coincides with the intersection of $X_{0}$ and $X_{m-1}$.
Since $X_{0}$ and $X_{m}$ form a distance normal pair and $d(X_{0},X_{m})=m$,
the face $A\cap X_{m}$ contains precisely $n-m$ vertices. 
So, there is the unique number $t\in \{m-1,\dots,n-1\}$
such that $x_{t}\not\in X_{m}$.
For every permutation $\delta$ on the set 
$\{m-1,\dots,n-1\}\setminus {t}$ we consider the vertex sequence
\begin{equation}\label{eq-zig}
x_{0},x_{1},\dots,x_{m-2},x_{t},x_{\delta(m-1)},\dots,\widehat{x_{t}},\dots,x_{\delta(n-1)}
\end{equation}
(the symbol $\,\widehat{}\,$ means that the corresponding element is omitted)
and denote by $F_{\delta}$ the flag obtained from this sequence.
This flag defines the zigzag 
$$Z_{\delta}=\{T^{i}(F_{\delta})\}_{i\in {\mathbb N}}.$$
Obviously, the first $n$ vertices in the $0$-shadow of $Z_{\delta}$ are \eqref{eq-zig}.
The next $m-1$ vertices are $x_{n},\dots,x_{n+m-2}$ which means that 
the $(n-1)$-shadow of $Z_{\delta}$ contains
the geodesic $X_{0},X_{1},\dots,X_{m-1}$.
The $m$-th element in the $(n-1)$-shadow of $Z_{\delta}$
is adjacent to $X_{m-1}$ and does not contain $x_{t}$.
This implies that it coincides with $X_{m}$. 
Therefore, the geodesic $X_{0},X_{1},\dots,X_{m}$
is contained in the $(n-1)$-shadow of $Z_{\delta}$.
Using the fact that $x_{i}$ is the unique vertex in $X_{i}\setminus X_{i+1}$ for $i\in \{0,1,\dots,m-2\}$,
$x_{t}$ is the unique vertex in $X_{m-1}\setminus X_{m}$ and
$$X_{0}\cap X_{m}=X_{0}\cap X_{1}\cap\dots\cap X_{m}= A\setminus\{x_{t}\},$$
we show that every zigzag whose $(n-1)$-shadow contains 
$X_{0},X_{1},\dots,X_{m}$ is of type $Z_{\delta}$.
As above, for distinct permutations $\delta$ and $\gamma$
the zigzags $Z_{\delta}$ and $Z_{\gamma}$ may be coincident.
For this reason, there are at most $(n-m)!$ such zigzags.

Now, we suppose that $m>n$.
For every $i\in \{0,1,\dots, m-n\}$ we have $$d(X_{i}, X_{i+n})=n$$ and  the facets $X_{i},X_{i+n}$ form a distance normal pair.
It was established above that there is the unique zigzag $Z_{i}$ whose $(n-1)$-shadow contains  the geodesic $X_{i},\dots, X_{i+n}$.
If $i<m-n$, then the geodesic $X_{i+1},\dots, X_{i+n}$ is contained in the $(n-1)$-shadows
of $Z_{i}$ and $Z_{i+1}$.
Since
$$d(X_{i+1},X_{i+n})=n-1,$$
 there is the unique zigzag whose $(n-1)$-shadow contains $X_{i+1},\dots, X_{i+n}$. 
Thus $Z_{i}$ coincides with $Z_{i+1}$ for every $i<n-m$.
This means that all zigzags $Z_{i}$ are coincident and we get the claim.

\subsection{$Z$-connectedness of non-maximal faces}
Two faces $X$ and $Y$ of the same non-maximal rank $k\ge 1$ 
are said to be {\it weakly adjacent} if their intersection is a $(k-1)$-face 
and there is no face containing both $X$ and $Y$.
This is a modification of the relation defined  for non-maximal singular subspaces of polar spaces
\cite[Subsection 4.6.2]{Pankov-book}.
Note that the non-maximal faces of $\beta_{n}$ considered in Example \ref{exmp4-2} 
are weakly adjacent.

\begin{lemma}\label{lemma4-3}
For every $k\in[n-2]$ any pair of weakly adjacent $k$-faces in $\Delta$
cannot be connected by a simple zigzag.
\end{lemma}

\begin{proof}
Let $\{x_{i}\}_{i\in {\mathbb N}}$ be the $0$-shadow 
of a zigzag connecting weakly adjacent $k$-faces $X$ and $Y$.
If the zigzag is simple, then there exists $i$ such that
$$X\cap Y=\{x_{i+1},\dots,x_{i+k}\},$$
one of $X,Y$ is $\{x_{i},\dots,x_{i+k}\}$ and the other is $\{x_{i+1},\dots,x_{i+k+1}\}$.
Since the rank $k$ is not maximal, 
we have $k+2\le n$ and
$$X\cup Y=\{x_{i},\dots,x_{i+k+1}\}$$
is a face which contradicts the fact that $X$ and $Y$ are weakly adjacent.
\end{proof}

A possible $z$-connectedness for two weakly adjacent faces is described 
in the following example.

\begin{exmp}{\rm
Let $\{x_{i}\}_{i\in {\mathbb N}}$ be the $0$-shadow of a zigzag.
The length $l$ of the zigzag is assumed to be sufficiently large.
Also, we suppose that the zigzag is not simple and there exist 
$i,j\in [l-1]$ and $k<n-1$
such that
$$i+k\le j,\;\;j+k\le l-1$$ 
and 
$$x_{i}=x_{j},x_{i+1}=x_{j+1},\dots,x_{i+k-1}=x_{j+k-1}.$$
If the vertices 
$x_{i-1},x_{i},\dots,x_{i+k-1},x_{j+k}$
do not form a face (this means that $j+k-i>n-1$),
then 
$$\{x_{i-1},x_{i},\dots,x_{i+k-1}\}\;\mbox{ and }\;
\{x_{j},\dots,x_{j+k-1},x_{j+k}\}$$
are weakly adjacent $k$-faces connected by our zigzag.
}\end{exmp}

\begin{lemma}\label{lemma4-4}
If $\Delta$ is $z$-simple and
any two edges of $\Delta$ are $z$-connected, then 
$\Delta$ is $3$-neigh\-borly.
\end{lemma}

\begin{proof}
By Lemma \ref{lemma4-3}, 
there exist no pairs of weakly adjacent edges, 
i.e. any two edges with a common vertex are adjacent.
If $x_{0},x_{1},\dots,x_{m}$ is a path in $\Gamma_{0}(\Delta)$,
then the edges $x_{0}x_{1}$ and $x_{1}x_{2}$ are adjacent
which implies that $x_{0},x_{2}$ are adjacent vertices and 
$x_{0},x_{2},\dots,x_{m}$ is a path in $\Gamma_{0}(\Delta)$.
Step by step, we show that the vertices $x_{0}$ and $x_{m}$ are adjacent.
Since the graph $\Gamma_{0}(\Delta)$ is connected, 
any two distinct vertices are adjacent.
Let $x_{1},x_{2},x_{3}$ be three distinct vertices of $\Delta$.
It was established above that they are mutually adjacent.
Then the edges $x_{1}x_{2}$ and $x_{2}x_{3}$ are adjacent and
we get the claim.
\end{proof}

The previous lemma can be generalized as follows.

\begin{prop}\label{prop4-1}
If $\Delta$ is $z$-simple and there is a non-zero number $k<n-1$ such that 
any two faces of the same non-zero rank $\le k$ are $z$-connected, 
then $\Delta$ is $(k+2)$-neighborly.
\end{prop}

\begin{proof}
The statement coincides with Lemma \ref{lemma4-4} if $k=1$.
Let $k\ge 2$.
Lemma \ref{lemma4-4} states that $\Delta$ is $3$-neighborly.
It follows from Lemma \ref{lemma4-3} that 
for every $i\in [k]$ two $i$-faces are adjacent if their intersection
is a $(i-1)$-face. 
Therefore, if $X$ is a $4$-element subset in the vertex set,
then any two distinct $3$-element subsets of $X$ are adjacent $2$-faces
and $X$ is a $3$-face. 
Step by step, we establish that every subset consisting of not greater than $k+2$
vertices is a face.
\end{proof}

Proposition \ref{prop4-1} together with Fact \ref{fact-neighborly}
give the following.

\begin{cor}\label{cor4-2}
Suppose that, as in Proposition \ref{prop4-1}, $\Delta$ is $z$-simple and
there is a non-zero number $k<n-1$ such that 
any two faces of the same non-zero dimension $\le k$ are $z$-connected.
If $k>\lfloor n/2\rfloor -2$, then $\Delta$ is the $n$-simplex.
\end{cor}

\subsection*{Acknowledgment}
The authors thank Mathieu Dutour Sikiri\'c for interest and discussion
and Christophe Hohlweg for remarks concerning Coxeter groups.
Also, they are grateful to anonymous referee who appealed their attention to \cite{Williams}.

\end{document}